\documentclass[10pt]{amsart}
\usepackage{amssymb}
\usepackage{verbatim}
\usepackage{pb-diagram}
\usepackage[dvipdfmx]{graphicx}
\theoremstyle{plain}
  \newtheorem{thm}{Theorem}[section]
  \newtheorem{lem}[thm]{Lemma}

  \newtheorem{conj}[thm]{Conjecture}

  \newtheorem*{obs*}{Observation}
\theoremstyle{definition}

\theoremstyle{remark}
  \newtheorem{rem}[thm]{Remark}
  
  \newtheorem*{ack}{Acknowledgments}
\newcommand{\Z}{\mathbb{Z}}
\newcommand{\C}{\mathbb{C}}

\newcommand{\Q}{\mathbb{Q}}
\newcommand{\Vol}{\operatorname{Vol}}
\newcommand{\CS}{\operatorname{CS}}

\newcommand{\Hom}{\operatorname{Hom}}
\newcommand{\cs}{\operatorname{cs}}
\newcommand{\Tr}{\operatorname{Tr}}
\newcommand{\Tor}{\operatorname{Tor}}

\newcommand{\floor}[1]{\lfloor#1\rfloor}
\renewcommand{\Re}{\operatorname{Re}}
\renewcommand{\Im}{\operatorname{Im}}
\newcommand{\Image}{\operatorname{Image}}
\newcommand{\Ker}{\operatorname{Ker}}
\newcommand{\para}{\xi}
\numberwithin{equation}{section}
\allowdisplaybreaks
\begin{document}
\title
{Representations and the colored Jones polynomial of a torus knot}
\author{Kazuhiro Hikami}
\address{
Department of Mathematics Education,
School of Natural and Living Sciences Education,
Naruto University of Education,
748, Nakashima, Takashima, Naruto-cho, Naruto-shi,
772-8502 Japan.}
\email{khikami@gmail.com}
\author{Hitoshi Murakami}
\address{
Department of Mathematics,
Tokyo Institute of Technology,
Oh-okayama, Meguro, Tokyo 152-8551, Japan
}
\email{starshea@tky3.3web.ne.jp}
\date{\today}
\begin{abstract}
We show that for a torus knot the $SL(2;\C)$ Chern--Simons invariants and the $SL(2;\C)$ twisted Reidemeister torsions appear in an asymptotic expansion of the colored Jones polynomial.
This suggests a generalization of the volume conjecture that relates the asymptotic behavior of the colored Jones polynomial of a knot to the volume of the knot complement.
\end{abstract}
\thanks{The authors are supported by Grant-in-Aid for Challenging Exploratory Research (21654053)}
\maketitle
\section{Introduction.}
In 1985, Jones introduced a knot invariant, the Jones polynomial, by using operator algebra \cite{Jones:BULAM385}.
It turns out to be a special case of a more general situation.
In fact for any simple Lie algebra $\mathfrak{g}$ and its irreducible representation $\rho$ one can define the quantum $(\mathfrak{g},\rho)$ invariant for knots (see for example \cite{Turaev:quantum}).
Then the Jones polynomial is regarded as the quantum $(\mathfrak{sl}(2;\C), V^2)$ invariant, where $V^2$ is the two-dimensional irreducible representation.
\par
Then, in 1989, Witten used Chern--Simons theory to describe the Jones polynomial in terms of path integral \cite{Witten:COMMP1989} and suggested quantum invariants for three-manifolds.
\par
Suppose that we are given a compact Lie group $G$ with Lie algebra $\mathfrak{g}$.
Let $K$ be a knot in the three-sphere $S^3$ and $V$ an irreducible representation of $G$.
Let $\mathcal{A}$ be the set of all $G$-connection on the trivial $G$-bundle over $S^3$.
For a $G$-connection $A$, define the Chern--Simons functional $L(A)$ to be
\begin{equation*}
  L(A)
  :=
  \frac{1}{4\pi}\int_{S^3}\Tr(A\wedge dA+\frac{2}{3}A\wedge A\wedge A).
\end{equation*}
Then Witten proposed the following Feynman path integral as a definition of the quantum invariant:
\begin{equation*}
  Z(S^3,K)
  :=
  \int_{\mathcal{A}}e^{\sqrt{-1}k\,L(A)}W_{V}(K;A)
  \mathcal{D}A,
\end{equation*}
where $W_{V}(K;A)$ is the Wilson loop, that is, the trace of the image in $V$ by the representation of the element in $G$ given by the parallel transport along $K$ using the connection $A$.
\par
If $G=SU(2)$ and $V$ is the $N$-dimensional irreducible representation, this defines the $N$-dimensional colored Jones polynomial $J_N\bigl(K;\exp(2\pi\sqrt{-1}/(k+2))\bigr)$.
\par
Since then many researches about these quantum invariants for knots and three-manifolds by both mathematicians and physicists.
\par
In 1995, Kashaev defined a yet another knot invariant $\langle{K}\rangle_N$ by using quantum dilogarithm \cite{Kashaev:MODPLA95}, where $N$ is an integer greater than one.
Moreover in \cite{Kashaev:LETMP97} he observed that for a few knots the limit $\log\left(\left|\langle{K}\rangle_N\right|\right)/N$ gives the hyperbolic volume of the knot complement $S^3\setminus{K}$.
He also conjectured this would be true for any hyperbolic knot.
Here a hyperbolic knot is a knot whose complement possesses a complete hyperbolic metric with finite volume.
\par
J.~Murakami and the second author proved that Kashaev's invariant is indeed a special value of the colored Jones polynomial \cite{Murakami/Murakami:ACTAM12001}.
More precisely, letting $J_N(K;q)$ be the colored Jones polynomial associated with the $N$-dimensional irreducible representation of the Lie algebra $\mathfrak{sl}(2;\C)$, we showed that $J_N\bigl(K;\exp(2\pi\sqrt{-1}/N)\bigr)$ is (essentially) equal to Kashaev's invariant.
We also generalized Kashaev's conjecture to the following conjecture (Volume Conjecture).
\begin{conj}[Volume Conjecture,\cite{Murakami/Murakami:ACTAM12001}]
\label{conj:VC}
For any knot, we have
\begin{equation*}
  2\pi\frac{\log\left|J_N\bigl(K;\exp(2\pi\sqrt{-1}/N)\bigr)\right|}{N}
  =
  \Vol(S^3\setminus{K}).
\end{equation*}
Here $\Vol$ is the simplicial volume (or the Gromov norm) \cite{Gromov:INSHE82} that is normalized so that it equals the sum of the hyperbolic volumes of the hyperbolic pieces in the JSJ decomposition \cite{Jaco/Shalen:MEMAM79,Johannson:1979} of the knot complement.
\par
Note that we normalize $J_N(K;q)$ so that $J_N(K;\text{unknot})=1$.
\end{conj}
The volume conjecture has been proved to be true for the following knots and links.
\begin{itemize}
\item
any torus knot by Kashaev and Tirkkonen \cite{Kashaev/Tirkkonen:ZAPNS2000},
\item
the torus link of type $(2,2m)$ by the first author \cite{Hikami:COMMP2004},
\item
the figure-eight knot by Ekholm (see for example \cite{Murakami:ALDT_VI}),
\item
the hyperbolic knot $5_2$ by Kashaev and Yokota,
\item
Whitehead doubles of the torus knot of type $(2,a)$ by Zheng \cite{Zheng:CHIAM22007},
\item
twisted Whitehead links by Zheng \cite{Zheng:CHIAM22007},
\item
the Borromean rings by Garoufalidis and L{\^e} \cite{Garoufalidis/Le:2005},
\item
Whitehead chains by van~der~Veen \cite{van_der_Veen:ACTMV2009}.
\end{itemize}
\par
What happens if we replace the $N$th root of unity $2\pi\sqrt{-1}/N$ with another complex parameter $\para/N$?
Yokota and the second author proved that for the figure-eight knot if $\para$ is close to $2\pi\sqrt{-1}$, then the limit gives the hyperbolic volume and the Chern--Simons invariant of the three-manifold obtained from $S^3$ by Dehn surgery along the figure-eight knot with coefficient given by $\para$ \cite{Murakami/Yokota:JREIA2007}.
\par
Note that the space of Dehn surgeries along a hyperbolic knot is complex one-dimensional \cite{Thurston:GT3M}, and the parameter $\para$ in the colored Jones polynomial can be regarded as a parameter of Dehn surgeries.
For a hyperbolic knot, the complete hyperbolic structure with finite volume is give by an irreducible representation (holonomy representation) of the fundamental group of its complement at the Lie group $PSL(2;\C)$.
Therefore it would be possible to use $\para$ to parameterize representations at $PSL(2;\C)$ or $SL(2;\C)$.
\par
In this paper we show that for torus knots we can relate the colored Jones polynomial evaluated at $\exp(\para/N)$ to representations of the fundamental group of a knot complement at $SL(2;\C)$.
Moreover by considering an asymptotic expansion of the colored Jones polynomial we can obtain the $SL(2;\C)$ Chern--Simons invariant and the twisted Reidemeister torsion both associated with the corresponding representation.
\par
The paper is organized as follows.
In Section~\ref{sec:character_variety} we describe the character variety of a torus knot, which is used to introduce the twisted Reidemeister torsion and the Chern--Simons invariant in Sections~\ref{sec:Reidemeister_torsion} and \ref{sec:Chern_Simons_invariant}, respectively.
In Section~\ref{sec:torus_knot} we calculate an asymptotic behavior of the colored Jones polynomial evaluated at $\exp(\para/N)$ for $N\to\infty$, and in Section~\ref{sec:topology} we give topological interpretations of its coefficients.
In the last section (Section~\ref{sec:speculation}) we give some speculation for general knots giving an observation about the figure-eight knot.
\begin{ack}
The authors would like to thank J.~Dubois, V.~Mu\~noz, E.~Witten, and Y.~Yamaguchi for helpful comments.
Thanks are also due to the organizers of the workshop ``Chern--Simons Gauge Theory: 20 years after'' held at the Max Planck Institute for Mathematics in August 2009 hosted by the Hausdorff Center for Mathematics.
\end{ack}
\section{$SL(2,\C)$ character variety of a torus knot.}
\label{sec:character_variety}
Let $T(a,b)$ be the torus knot of type $(a,b)$, where $a$ and $b$ are coprime positive integers.
Throughout this paper we assume that $b$ is odd.
Let $X\bigl(T(a,b)\bigr)$ be the character variety of $\pi_1\bigl(S^3\setminus T(a,b)\bigr)$ of representations of $\pi_1\bigl(S^3\setminus T(a,b)\bigr)$ at $SL(2,\C)$ \cite{Culler/Shalen:ANNMA21983}.
So two homomorphisms from $\pi_1(S^3\setminus{T(a,b)})$ to $SL(2,\C)$ are regarded as equivalent if and only if they have the same trace.
\par
We will describe $X\bigl(T(a,b)\bigr)$ following \cite{Munoz:REVMC2009}.
\par
Let $\langle x,y\mid x^a=y^b\rangle$ be a presentation of $\pi_1(S^3\setminus{T(a,b)})$.
\par
There is a unique reducible component, which is homeomorphic to $\C$ by assigning $[\varphi_t]\in X\bigl(T(a,b)\bigr)$ to $t+t^{-1}\in\C$, where $\varphi_t$ sends $x$ to $\begin{pmatrix}t^b&0\\0&t^{-b}\end{pmatrix}$ and $y$ to $\begin{pmatrix}t^a&0\\0&t^{-a}\end{pmatrix}$.
Here square brackets mean the class of a representation in the character variety.
\par
The irreducible characters decompose into $(a-1)(b-1)/2$ components and each of them is homeomorphic to $\C$.
They are indexed by a pair of integers $(\alpha,\beta)$ such that $1\le\alpha\le a-1$, $1\le\beta\le b-1$, and that $\alpha\equiv\beta\pmod{2}$.
See also \cite[Theorem~1]{Klassen:TRAAM1991}, \cite[Theorem~2]{Dubois/Kashaev:MATHA2007}.
A representation with index $(\alpha,\beta)$ sends $x$ to an element with trace $2\cos(\pi\alpha/a)$ and $y$ to one with trace $2\cos(\pi\beta/b)$.
\par
The closure of the component indexed by $(\alpha,\beta)$ intersects the reducible component in two points $\left[\varphi_{\exp\bigl(k_1\pi\sqrt{-1}/(ab)\bigr)}\right]$ and $\left[\varphi_{\exp\bigl(k_2\pi\sqrt{-1}/(ab)\bigr)}\right]$, where 
\begin{align*}
  k_1&\equiv\alpha\pmod{a},\quad
  k_1\equiv-\beta\pmod{b},
  \\
  k_2&\equiv\alpha\pmod{a},\quad
  k_2\equiv\beta\pmod{b}.
\end{align*}
Note that $k_1$ and $k_2$ are uniquely determined by the formulas above.
\par
\begin{rem}
Our pair $(k_1,k_2)$ is different from Dubois and Kashaev's pair $(k_-,k_+)$ \cite[Theorem~2]{Dubois/Kashaev:MATHA2007}.
They choose $k_-$ and $k_+$ so that $k_-\equiv k_+\pmod{2}$.
\end{rem}
Conversely, given a positive integer $k$ that is not a multiple of neither $a$ nor $b$, we can define a pair $(\alpha,\beta)$ such that $1\le\alpha\le a-1$, $1\le\beta\le b-1$, and $\alpha\equiv\beta\pmod{2}$ as follows:
Define $\alpha$ to be the integer that is congruent modulo $a$ to $k$ with $1\le\alpha\le a-1$, $\beta'$ to be the integer that is congruent modulo $b$ to $k$ with $1\le\beta'\le b-1$.
If $\alpha\equiv\beta'\pmod2$ then put $\beta:=\beta'$, and if $\alpha\not\equiv\beta'\pmod2$ then put $\beta:=b-\beta'$.
Note that since we assume that $b$ is odd $\beta$ always has the same parity as $\alpha$.
\begin{rem}\label{rem:def_alpha_beta}
If $k$ defines $(\alpha,\beta)$ as above, then the pair $(k_1,k_2)$ defined by $(\alpha,\beta)$ is either $(k,-k)$ or $(-k,k)$ $\pmod{ab}$.
So the assignment of $k\in\{n\in\Z\mid 1\le n\le ab-1,a\nmid n,b\nmid n\}$ to $(\alpha,\beta)\in\{l\in\Z\mid 1\le l\le a-1\}\times\{m\in\Z\mid 1\le m\le b-1\}$ is a two-to-one correspondence.
\par
Note that in either case $\sin^2(\alpha\pi/a)\sin^2(\beta\pi/b)$, which appears in the twisted Reidemeister torsion (see \S\ref{sec:topology}), does not depend on the definition that we use and equals $\sin^2(k\pi/a)\sin^2(k\pi/b)$.
\end{rem}
\section{Twisted Reidemeister torsion for a knot.}
\label{sec:Reidemeister_torsion}
Let $K$ be a knot in $S^3$ and $\rho$ a representation of $\pi_1(S^3\setminus{K})$ at $SL(2;\C)$.
Put $C^{\ast}\bigl(S^3\setminus{K};\rho\bigr):=\Hom_{\Z[\pi_1(S^3\setminus{K})]}(C_{\ast}(\widetilde{S^3\setminus{K}};\Z),\mathfrak{sl}(2;\C))$.
Here $\widetilde{S^3\setminus{K}}$ is the universal cover of $S^3\setminus{K}$, $C_{\ast}(\widetilde{S^3\setminus{K}};\Z)$ is regarded as a $\Z[\pi_1(S^3\setminus{K})]$-module by the action of the deck transformation and $\mathfrak{sl}(2;\C)$ is regarded as a $\Z[\pi_1(S^3\setminus{K})]$-module via the adjoint representation.
\par
Let $\{0\}\to C^{0}\xrightarrow{d^0}C^{1}\xrightarrow{d^1}C^{2}\xrightarrow{d^3}C^{3}\to\{0\}$ be the corresponding cochain complex, where $C^{i}:=C^{i}(S^3\setminus{K};\rho)$ and $d^i$ is the coboundary map induced by the boundary map of $C_{\ast}(\widetilde{S^3\setminus{K}};\Z)$.
Put $B^{i}:=\Image(d^{i-1})\subset C^{i}$, $Z^{i}:=\Ker(d^i)\subset C^{i}$, and $H^i:=Z^{i}/B^{i}$.
\par
We choose bases $\mathbf{c}^i$ of $C^i$ and $\mathbf{h}^i$ of $H^i$.
Let $\tilde{\mathbf{h}}^i\subset Z^i$ be a lift of $\mathbf{h}^i$ and $\mathbf{b}^i\subset C^i$ be a set of elements such that $d^i(\mathbf{b}^i)$ forms a basis of $B^{i+1}$.
Since $B^{i+1}\cong C^i/Z^i$ and $H^i\cong Z^i/B^i$, the set $d^{i-1}(\mathbf{b}^{i-1})\cup\tilde{\mathbf{h}}^i\cup\mathbf{b}^i$ forms a basis of $C^i$.
Define $\left[\left(d^{i-1}(\mathbf{b}^{i-1})\cup\tilde{\mathbf{h}}^i\cup\mathbf{b}^i\right)/\mathbf{c}^i\right]$ to be the determinant of the change-of-basis matrix from $\mathbf{c}^i$ to $d^{i-1}(\mathbf{b}^{i-1})\cup\tilde{h}^i\cup\mathbf{b}^i$.
\par
Then the Reidemeister torsion (\cite{Reidemeister:ABHMS1935}, \cite{Franz:JREIA1935}, \cite{deRham:1936}, \cite{Milnor:ANNMA21962}, \cite{Turaev:USPMN1986}) with respect to $\mathbf{c}^i$ and $\mathbf{h}^i$ is defined to be
\begin{equation}\label{eq:torsion}
  \Tor(C^{\ast},\mathbf{c}^{\ast},\mathbf{h}^{\ast})
  :=
  \prod_{i=0}^{n}
  \left[
    \left(
      d^{i-1}(\mathbf{b}^{i-1})\cup\tilde{\mathbf{h}}^i\cup\mathbf{b}^i
    \right)
    /\mathbf{c}^i
  \right]^{(-1)^{i+1}}.
\end{equation}
It is known that $\Tor(C^{\ast},\mathbf{c}^{\ast},\mathbf{h}^{\ast})$ does not depend on the choice of $\mathbf{b}^i$ and $\tilde{\mathbf{h}}^i$.
It is also known that up to sign it depends only on the choice of $\mathbf{h}^{\ast}$.
(We need a cohomological orientation to define the sign but in this paper we do not need it.
See \cite{Turaev:USPMN1986} and \cite{Dubois:CANMB2006} for details.)
\par
To define a basis $\mathbf{h}^{i}$ of $H^{i}$ we need to choose a simple closed curve on $\partial{E_K}$, where $E_K:=S^3\setminus{\operatorname{Int}(N(K))}$ with $N(K)$ the regular neighborhood of $K$ in $S^3$.
\par
An irreducible representation $\rho$ is called $\gamma$-regular \cite{Porti:MAMCAU1997,Dubois:CANMB2006} for a simple closed curve $\gamma\subset\partial{E_K}$ if the following two conditions are satisfied.
\begin{itemize}
\item
The homomorphism $i^{\ast}\colon H^1(E_K;\rho)\to H^1(\gamma;\rho)$ induced by the inclusion $i\colon\gamma\hookrightarrow{E_K}$ is injective.
Note that $H^{\ast}(E_K;\rho)$ is isomorphic to $H^{\ast}(S^3\setminus{K};\rho)$.
\item
If $\Tr\left(\rho\left(\pi_1(\partial{E_K})\right)\right)\subset\{\pm2\}$, then $\rho(\gamma)$ is not $\pm{I}$, where $I$ is the identity matrix.
\end{itemize}
If $\rho$ is $\gamma$-regular, then $\dim H^1(S^3\setminus{K};\rho)=\dim H^2(S^3\setminus{K};\rho)=1$ and $\dim H^i(S^3\setminus{K};\rho)=0$ for $i\ne1,2$ \cite[Lemma~2]{Dubois:CANMB2006}.
So to define the Reidemeister torsion for a $\gamma$-regular representation $\rho$ we only need to choose a non-zero element of $H^1(S^3\setminus{K};\rho)$ and a non-zero element of $H^2(S^3\setminus{K};\rho)$.
We use $\gamma$ to define such an element of $H^1(S^3\setminus{K};\rho)=H^1(E_K;\rho)$ and the fundamental class $[\partial{E_K}]\in H^2(\partial{E_K};\Z)$ to define such an element of $H^2(S^3\setminus{K};\rho)$ (for details, see \cite[\S~3]{Dubois/Kashaev:MATHA2007} for example).
\par
Therefore given a simple closed curve $\gamma\subset\partial{E_K}$ such that $\rho$ is $\gamma$-regular one can define the Reidemeister torsion (\cite{Porti:MAMCAU1997}, \cite{Dubois:CANMB2006}) by \eqref{eq:torsion} up to sign.
It is denoted by $\mathbb{T}_{\gamma}^{K}(\rho)$ and called the twisted Reidemeister torsion.
\par
It is known that for a torus knot, any irreducible representation is both $\mu$-regular and $\lambda$-regular, where $\mu$ is the meridian, a loop that goes around the knot so that its linking number with the knot is one, and $\lambda$ is the preferred longitude, a loop that goes along the knot so that its linking number with the knot is zero \cite[Example~1]{Dubois:CANMB2006}.
It is also known that for a hyperbolic knot $K$, then an irreducible representation that defines a hyperbolic Dehn surgery is $\gamma$-regular, where $\gamma$ is the simple closed curve on $\partial{E_K}$ along which the surgery is performed \cite{Porti:MAMCAU1997}.
\section{Chern--Simons invariant for a knot.}
\label{sec:Chern_Simons_invariant}
We follow \cite{Kirk/Klassen:COMMP93} to define the $SL(2;\C)$ Chern--Simons invariant.
\par
For a closed three-manifold $M$, one can define the $SL(2;\C)$ Chern--Simons function $\cs_M\colon X(M)\to\C\pmod\Z$, where $X(M)$ is the $SL(2;\C)$ character variety of $M$.
Let $A$ be an $\mathfrak{sl}(2;\C)$-valued $1$-form on $M$ with $dA+A\wedge A=0$.
Then $A$ defines a flat connection of $M\times SL(2;\C)$ and so one can define a representation $\rho\colon\pi_1(M)\to SL(2;\C)$ by holonomy.
The Chern--Simons function is defined to be
\begin{equation*}
  \cs_M([\rho])
  :=
  \frac{1}{8\pi^2}
  \int_{M}
  \Tr(A\wedge dA+\frac{2}{3}A\wedge A\wedge A)
  \in\C\pmod\Z,
\end{equation*}
where $[\rho]$ is the class of $\rho$ in $X(M)$.
\par
Now we assume that $M$ has a boundary which is homeomorphic to a torus.
Denote by $X(\partial{M})$ the $SL(2;\C)$ character variety of the boundary $\partial{M}$.
\par
We define $E(\partial{M})$ as the quotient space of $\Hom(\pi_1(\partial{M}),\C)\times\C^{\ast}$ by a group $G$, where
\begin{equation*}
  G
  :=
  \langle
    X,Y,B\mid
    XYX^{-1}Y^{-1}=XBXB=YBYB=B^2=1
  \rangle
\end{equation*}
and it acts on $\Hom(\pi_1(\partial{M}),\C)\times\C^{\ast}$ by
\begin{equation}\label{eq:coordinate_change}
\begin{split}
  X\cdot(s,t;z)&:=(s+1,t;z\exp(-8\pi\sqrt{-1}t)), \\
  Y\cdot(s,t;z)&:=(s,t+1;z\exp(8\pi\sqrt{-1}s)), \\
  B\cdot(s,t;z)&:=(-s,-t;z).
\end{split}
\end{equation}
Here a pair $(s,t)$ is identified with the element $s\gamma^{\ast}+t\delta^{\ast}\in\Hom(\pi_1(\partial{M}),\C)$ with a fixed basis $(\gamma,\delta)$ of $\pi_1(\partial{M})\cong\Z\oplus\Z$.
Then $E(\partial{M})$ becomes a $\C^{\ast}$-bundle over $X(\partial{M})$.
Note that $X(\partial M)$ is identified with $\Hom(\pi_1(\partial M),\C)/G$ via the quotient map $q\colon\Hom(\pi_1(\partial{M}),\C)\to\Hom(\pi_1(\partial{M}),SL(2;\C))$ defined by
\begin{align*}
  q(\kappa)
  :=
  \left[
    \gamma
    \mapsto
    \begin{pmatrix}
      e^{2\pi\sqrt{-1}\kappa(\gamma)}&0 \\
      0&e^{-2\pi\sqrt{-1}\kappa(\gamma)}
    \end{pmatrix}
  \right]
\end{align*}
for $\gamma\in\pi_1(\partial{M})$.
\par
The Chern--Simons function  $\cs_M$ in this case is defined to be a map from $X(M)$ to $E(\partial{M})$ such that $p\circ\cs_M=i^{\ast}$, where $p\colon E(\partial{M})\to X(\partial{M})$ is the projection and $i^{\ast}\colon X(M)\to X(\partial{M})$ is induced from the inclusion map $i\colon \partial{M}\hookrightarrow M$.
\begin{equation*}
\begin{diagram}
  \node[2]{E(\partial{M})}\arrow{s,r}{p} \\
  \node{X(M)}\arrow{ne,t}{\cs_M}\arrow{e,t}{i^{\ast}}\node{X(\partial{M})}
\end{diagram}
\end{equation*}
See \cite[\S~3]{Kirk/Klassen:COMMP93} for the precise definition.
\par
If we have another three-manifold $M'$ with toral boundary, we can construct a closed three-manifold $M\cup_{\partial}M'$ by identifying $\partial{M}$ with $-\partial{M}'$.
Given a representation $\rho\colon M\cup_{\partial}M'\to SL(2;\C)$, the Chern--Simons invariant $\cs_{M\cup_{\partial}M'}([\rho])$ is given by $zz'$ if $\cs_{M}\left(\left[\rho\bigr|_{M}\right]\right)=[s,t;z]$ and $\cs_{M'}\left(\left[\rho\bigr|_{M'}\right]\right)=[s,t;z']$, where $\rho\bigr|_{M}$ and $\rho\bigr|_{M'}$ are the restrictions of $\rho$ to $M$ and $M'$ respectively.
Note that we use the same basis for $\pi_1(\partial{M})$ and $\pi_1(-\partial{M}')$.
\par
Suppose that $M$ is the complement of the interior of the regular neighborhood of a knot $K$ in $S^3$.
Let $\rho$ be a representation sending the meridian $\mu$ and the longitude $\lambda$ to the elements (up to conjugation) shown below.
\begin{align*}
  \rho(\mu)
  &=
  \begin{pmatrix}
    \exp(u/2) & \ast \\
    0         & \exp(-u/2)
  \end{pmatrix},
  \\
  \rho(\lambda)
  &=
  \begin{pmatrix}
    \exp(v/2) & \ast \\
    0         & \exp(-v/2)
  \end{pmatrix}.
\end{align*}
We also assume that the elements in $\Hom(\pi_1(\partial{M}),\C)$ sending $\mu$ to $u$ and $\lambda$ to $v$ form a basis.
Then we introduce the function $\CS_{u,v}([\rho])$ as follows.
\begin{equation*}
  \cs_{M}([\rho])
  =
  \left[
    \frac{u}{4\pi\sqrt{-1}},
    \frac{v}{4\pi\sqrt{-1}};
    \exp\left(\frac{2}{\pi\sqrt{-1}}\CS_{u,v}([\rho])\right)
  \right].
\end{equation*}
Note that $\CS_{u,v}([\rho])$ is defined modulo $\pi^2\Z$ and that it depends on lifts $(u,v)$ of $(\exp(u/2),\exp(v/2))$.
\begin{rem}
Note that we are using the $PSL(2;\C)$ normalization described in \cite[P.~543]{Kirk/Klassen:COMMP93}.
So our $\cs_{u,v}([\rho])$ is $-4$ times $f(u)$ in \cite{Neumann/Zagier:TOPOL85,Murakami:ACTMV2008}, and Kirk and Klassen's (and so Yoshida's \cite{Yoshida:INVEM85}) $f(u)$ is $\pi\sqrt{-1}/2\times\cs_{u,v}(u)$.
\end{rem}
\section{An asymptotic behavior of the colored Jones polynomial of a torus knot.}
\label{sec:torus_knot}
In this section we give asymptotic expansions of the colored Jones polynomial of a torus knot.
\par
Let $J_N(K;q)$ be the $N$-dimensional colored Jones polynomial of a knot $K$.
We normalize it so that $J_N(\text{unknot};q)=1$.
So using Witten's formulation $J_N(K;q)=Z(S^3,K)/Z(S^3,\text{unknot})$ with $G=SU(2)$ and $V$ is the $N$-dimensional irreducible representation.
Note that $J_2(K;q)=V_K(q^{-1})$ for any knot $K$, where $V_K(q)$ is the original Jones polynomial \cite{Jones:BULAM385}.
\par
Let $\Delta(K;t)$ be the Alexander polynomial for a knot $K$.
We normalize it so that $\Delta(K;t)=\Delta(K;t^{-1})$ and $\Delta(K;1)=1$.
\par
Now we consider the torus knot $T(a,b)$.
For a complex parameter $z$, we put
\begin{equation*}
  \tau_{a,b}(z)
  :=
  \frac{2\sinh(z)}{\Delta\bigl(T(a,b);e^{2z}\bigr)}.
\end{equation*}
Since it is well-known that
\begin{equation*}
  \Delta\bigl(T(a,b);t\bigr)
  =
  \frac{\left(t^{ab/2}-t^{-ab/2}\right)\left(t^{1/2}-t^{-1/2}\right)}
       {\left(t^{a/2}-t^{-a/2}\right)\left(t^{b/2}-t^{-b/2}\right)},
\end{equation*}
we have
\begin{equation*}
  \tau_{a,b}(z)
  =
  \frac{2\sinh(az)\sinh(bz)}{\sinh(abz)}.
\end{equation*}
Note that $(t^{1/2}-t^{-1/2})/\Delta(K;t)$ can be regarded as the (abelian) Reidemeister torsion (\cite[Theorem~4]{Milnor:ANNMA21962},\cite[Theorem~1.1.2]{Turaev:USPMN1986}).
Since we use cohomology to define the torsion but Milnor and Turaev use homology,our torsion is the inverse of theirs.
\par
Let $\mathcal{P}$ be the set of poles of $\tau_{a,b}(z)$, that is, we put
\begin{equation*}
  \mathcal{P}
  :=
  \left\{
    \dfrac{k\pi\sqrt{-1}}{ab}
    \Biggm|
    k\in\Z,
    a\nmid k,
    b\nmid k
  \right\}.
\end{equation*}
\par
We also put
\begin{equation*}
  A_{k}(\para;N)
  =
  \sqrt{-\pi}
  \exp\left(S_k(\para)\frac{N}{\para}\right)
  \left(\frac{N}{\para}\right)^{1/2}
  \left(T_k\right)^{1/2},
\end{equation*}
where
\begin{align*}
  S_k(\para)
  &:=
  \frac{-\left(2k\pi\sqrt{-1}-ab\para\right)^2}{4ab}
  \\
\intertext{and}
  T_k
  &:=
  \frac{16\sin^2(k\pi/a)\sin^2(k\pi/b)}{ab}.
\end{align*}
\par
We would like to know an asymptotic behavior of $J_N\bigl(T(a,b);\exp(\para/N)\bigr)$ for large $N$.
\par
The case where $\para=2\pi\sqrt{-1}$ corresponds to the volume conjecture (Conjecture~\ref{conj:VC}).
In this case, Kashaev and Tirkkonen \cite{Kashaev/Tirkkonen:ZAPNS2000} proved the following asymptotic expansion.
\begin{equation*}
\begin{split}
  &J_N\bigl(T(a,b);\exp(2\pi\sqrt{-1}/N)\bigr)
  \\
  \sim&
  e^{(ab-a/b-b/a)\pi\sqrt{-1}/(2N)}
  \\
  &\times
  \left(
    \frac{\pi^{3/2}}{2ab}
    \left(\frac{N}{\para}\right)^{3/2}
    \sum_{k=1}^{ab-1}
    (-1)^{k+1}k^2
    \exp\left(S_k(\para)\frac{N}{\para}\right)
    \left(T_k\right)^{1/2}
  \right.
  \\
  &\qquad
  \left.
    +
    \frac{1}{4}
    \sum_{j=1}^{\infty}
    \frac{a_j}{j!}
    \left(\frac{\para}{4abN}\right)^{j-1}
  \right),
\end{split}
\end{equation*}
where $a_l$ is the $2l$-th derivative of $2z\sinh{z}/\Delta\bigl(T(a,b);e^{2z}\bigr)=z\tau_{a,b}(z)$ at $z=0$.
For a relation to characters of conformal field theory, see \cite{Hikami/Kirillov:PHYLB2003}.
See also \cite{Dubois/Kashaev:MATHA2007} for a topological interpretation of this expansion.
\par
When $\para$ is not an integer multiple of $2\pi\sqrt{-1}$, we have the following theorem.
\begin{thm}\label{thm:asymptotic}
Let $\para$ be a complex number that is not an integral multiple of $2\pi\sqrt{-1}$.
We also assume that $\Im\para\ge0$ for simplicity.
\par
If $\para/2\not\in\mathcal{P}$, then we have
\begin{equation}\label{eq:not_pole}
\begin{split}
  &J_N\bigl(T(a,b);\exp(\para/N)\bigr)
  \\
  \sim&
  \frac{e^{(ab-a/b-b/a)\para/(4N)}}{2\sinh(\para/2)}
  \left(
    \tau_{a,b}(\para/2)
    +
    \sum_{j=1}^{\infty}
    \frac{\tau_{a,b}^{(2j)}(\para/2)}{j!}
    \left(\frac{\para}{4abN}\right)^{j}
  \right)
\end{split}
\end{equation}
when $\Re{\para}>0$ and
\begin{equation}
\begin{split}
  &J_N\bigl(T(a,b);\exp(\para/N)\bigr)
  \\
  \sim&
  \frac{e^{(ab-a/b-b/a)\para/(4N)}}{2\sinh(\para/2)}
  \left(
    \tau_{a,b}(\para/2)
    +
    \sum_{k=1}^{\floor{ab|\para|/(2\pi)}}
    (-1)^{k+1}A_k(\para;N)
    +
    \sum_{j=1}^{\infty}
    \frac{\tau_{a,b}^{(2j)}(\para/2)}{j!}
    \left(\frac{\para}{4abN}\right)^{j}
  \right)
\end{split}
\end{equation}
when $\Re{\para}\le0$ as $N\to\infty$, where $\tau_{a,b}^{(2j)}(\para/2)$ is the $(2j)$th derivative of $\tau_{a,b}(z)$ at $z=\para/2$ and $\floor{x}$ means the largest integer that does not exceed $x$.
\par
If $\para/2\in\mathcal{P}$ $($and it is not an integer multiple of $\pi\sqrt{-1}$$)$, then we have
\begin{equation}\label{eq:pole}
\begin{split}
  &
  J_N\bigl(T(a,b);\exp(\para/N)\bigr)
  \\
  \sim&
  \frac{e^{(ab-a/b-b/a)\para/(4N)}}{2\sinh(\para/2)}
  \left(\vphantom{\left(\frac{\para}{4abN}\right)^{j}}
    \tau_{a,b}^{(0)}(\para/2)
    +
    \frac{1}{2}(-1)^{ab|\para|/(2\pi)}A_{ab|\para|/(2\pi)}(\para;N)
  \right.
  \\
  &\quad
  \left.
    +
    \sum_{k=1}^{ab|\para|/(2\pi)-1}
    (-1)^{k+1}A_{k}(\para;N)
    +
    \sum_{j=1}^{\infty}
    \frac{\tau_{a,b}^{(2j)}(\para/2)}{j!}
    \left(\frac{\para}{4abN}\right)^{j}
  \right)
\end{split}
\end{equation}
as $N\to\infty$, where $\tau_{a,b}^{(0)}(\para/2)$ is the constant term of the Laurent expansion of $\tau_{a,b}(z)$ around $z=\para/2$.
\end{thm}
\begin{rem}
If $\Re\para>0$, or $\Re\para\le0$ and $|\para|<2\pi/(ab)$, then $J_N\bigl(T(a,b);\exp(\para/N)\bigr)$ converges to $\tau_{a,b}(\para/2)/(2\sinh(\para/2))=1/\Delta\bigl(T(a,b);\exp{\para}\bigr)$.
Otherwise it diverges.
See Figure~\ref{fig}.
\par
Note that Garoufalidis and L{\^e} proved that for any knot $K$, $J_N\bigl(K;\exp(\para/N)\bigr)$ converges to $1/\Delta(K;\exp\para)$ when $|\para|$ is small enough  \cite{Garoufalidis/Le:aMMR}.
\begin{figure}[h]
\begin{center}
\includegraphics{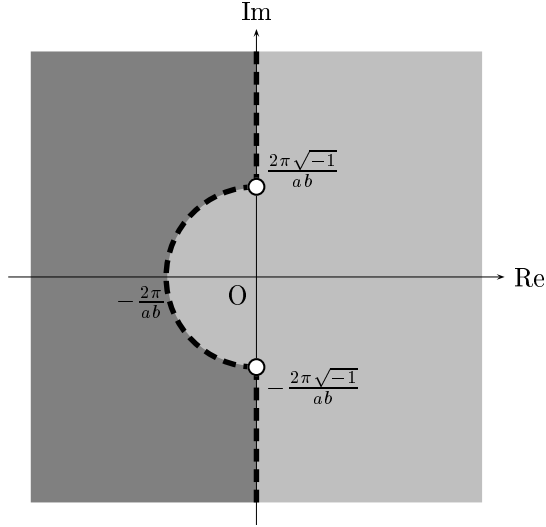}
\caption{
The colored Jones polynomial converges in the light gray area, and diverges in the gray area including the dashed lines and semicircle except for $\mathcal{P}$ indicated by the white circles.}
\label{fig}
\end{center}
\end{figure}
\end{rem}
\begin{proof}[Proof of Theorem~\ref{thm:asymptotic} for $\para$ with non-zero real part.]
We first prove Theorem~\ref{thm:asymptotic} where $\Re\para\ne0$.
Recall that we assume $\Im\para\ge0$.
\par
In \cite{Kashaev/Tirkkonen:ZAPNS2000}, Kashaev and Tirkkonen proved that $J_N\bigl(T(a,b);\exp(\para/N)\bigr)$ is given by the following integral.
\begin{equation*}
  J_N\bigl(T(a,b);\exp(\para/N)\bigr)
  =
  \Phi_{a,b,\para}(N)
  \int_{C}e^{abN(-z^2/\para+z)}\tau_{a,b}(z)\,dz,
\end{equation*}
where
\begin{equation*}
  \Phi_{a,b,\para}(N)
  :=
  \frac{1}{2\sinh(\para/2)}
  \sqrt{\frac{abN}{\pi\para}}
  e^{-abN\para/4+(ab-a/b-b/a)\para/(4N)}
\end{equation*}
and $C$ is the line passing through the origin with slope $\tan(\varphi)$, where $\varphi$ is chosen so that $(\arg\para)/2-\pi/4<\varphi<(\arg\para)/2+\pi/4$.
Note that this is to make the integral converges.
\par
Let $C_{\para}$ be the line that is parallel to $C$ and passes through $\para/2$ that is the critical point of the exponent of the integrand.
Then we have
\begin{equation*}
\begin{split}
  &
  J_N\bigl(T(a,b);\exp(\para/N)\bigr)
  \\
  =&
  \Phi_{a,b,\para}(N)
  \left(\vphantom{\sum_{k}}
    \int_{C_{\para}}e^{abN(-z^2/\para+z)}\tau_{a,b}(z)\,dz
  \right.
  \\
  &+
  \left.
    2\pi\sqrt{-1}
    \sum_{k}
    \operatorname{Res}
    \left(
      e^{abN(-z^2/\para+z)}
      \tau_{a,b}(z);z=k\pi\sqrt{-1}/(ab)
    \right)
  \right)
  \\
  =&
  \Phi_{a,b,\para}(N)
  \left(\vphantom{\sum_{k}}
    \int_{C_{\para}}e^{abN(-z^2/\para+z)}\tau_{a,b}(z)\,dz
  \right.
  \\
  &+
  \left.
    2\pi\sqrt{-1}
    \sum_{k}
    (-1)^{k+1}
    \frac{2\sin(k\pi/a)\sinh(k\pi/b)}{ab}
    \exp\left(N\left(\frac{k^2\pi^2}{ab\para}+k\pi\sqrt{-1}\right)\right)
  \right)
\end{split}
\end{equation*}
in a similar way to \cite{Murakami:INTJM62004}, where $k$ runs over integers such that $k\pi\sqrt{-1}/(ab)$ is between $C$ and $C_{\para}$.
\par
First we calculate the asymptotic expansion of the integral.
Putting $w:=z-\para/2$, we have
\begin{equation*}
\begin{split}
  \int_{C_{\para}}e^{abN(-z^2/\para+z)}\tau_{a,b}(z)\,dz
  &=
  \int_{C}e^{abN\big(-(w+\para/2)^2/\para+(w+\para/2)\bigr)}
  \tau_{a,b}(w+\para/2)\,dw
  \\
  &=
  e^{abN\para/4}
  \int_{C}e^{-abNw^2/\para}\tau_{a,b}(w+\para/2)\,dw
  \\
  &=
  e^{abN\para/4}e^{\varphi\sqrt{-1}}
  \int_{-\infty}^{\infty}
  e^{-abN e^{2\varphi\sqrt{-1}}t^2/\para}
  \tau_{a,b}
  \left(te^{\varphi\sqrt{-1}}+\para/2\right)
  \,dt.
\end{split}
\end{equation*}
We use Watson's Theorem, a special case of the steepest descent method, to obtain the asymptotic expansion of this integral.
\begin{thm}[\cite{Marsden/Hoffman:Complex_Analysis}(Theorem~7.2.7)]
Let $g(z)$ be analytic and bounded on a domain containing the real axis.
Set $f(z):=\int_{-\infty}^{\infty}e^{-zy^2/2}g(y)\,dy$ for $z$ real.
Then
\begin{equation*}
  f(z)
  \sim
  \sqrt{\frac{2\pi}{z}}
  \left(
    a_0+\frac{a_2}{z}+\frac{3!! a_4}{z^2}+
    \frac{5!! a_6}{z^3}+\cdots
  \right)
\end{equation*}
as $z\to\infty$, $\arg{z}=0$, where $g(z)=\sum_{n=0}^{\infty}a_nz^n$ near zero.
\end{thm}
Note that this theorem also holds for $z$ with fixed $\arg{z}$ with $\Re{z}>0$.
\par
Since $\arg\left(e^{2\varphi\sqrt{-1}}/\para\right)=2\varphi-\arg{\para}$, we see $\Re\left(abN e^{2\varphi\sqrt{-1}}/\para\right)>0$.
Therefore we have the following asymptotic expansion.
\begin{equation*}
\begin{split}
  &
  \int_{-\infty}^{\infty}
  e^{-abN e^{2\varphi\sqrt{-1}}t^2/\para}
  \tau_{a,b}
  \left(te^{\varphi\sqrt{-1}}+\para/2\right)
  \,dt
  \\
  \sim&
  \sqrt{\frac{\pi\para}{abN e^{2\varphi\sqrt{-1}}}}
  \left(
    \sum_{j=0}^{\infty}
    \frac{(2j-1)!!c_{2j}}{(2abNe^{2\varphi\sqrt{-1}}/\para)^{j}}
  \right)
\end{split}
\end{equation*}
as $N\to\infty$, where $c_{2j}$ is the coefficient of $t^{2j}$ in the Taylor expansion of $\tau_{a,b}\left(te^{\varphi\sqrt{-1}}+\para/2\right)$ around $t=0$.
Since $c_{2j}=\dfrac{e^{2j\varphi\sqrt{-1}}}{(2j)!}\tau_{a,b}^{(2j)}(\para/2)$ with $\tau_{a,b}^{(2j)}(\para/2)$ the $2j$th derivative of $\tau_{a,b}(z)$ at $\para/2$, we have
\begin{equation*}
\begin{split}
  &
  \int_{C_{\para}}e^{abN(-z^2/\para+z)}\tau_{a,b}(z)\,dz
  \\
  \sim&
  e^{abN\para/4}e^{\varphi\sqrt{-1}}
  \sqrt{\frac{\pi\para}{abN e^{2\varphi\sqrt{-1}}}}
  \left(
    \sum_{j=0}^{\infty}
    \frac{(2j-1)!!\tau_{a,b}^{(2j)}(\para/2)}{(2j)!(2abN/\para)^{j}}
  \right)
  \\
  =&
  e^{abN\para/4}
  \sqrt{\frac{\pi\para}{abN}}
  \left(
    \sum_{j=0}^{\infty}
    \frac{\tau_{a,b}^{(2j)}(\para/2)}{j!}
    \left(\frac{\para}{4abN}\right)^{j}
  \right).
\end{split}
\end{equation*}
\par
Since $\tau_{a,b}^{(0)}(\para/2)=\tau_{a,b}(\para/2)$, we finally have the following asymptotic expansion.
\begin{equation*}
\begin{split}
  &J_N\bigl(T(a,b);\exp(\para/N)\bigr)
  \\
  \sim&
  \frac{e^{(ab-a/b-b/a)\para/(4N)}}{2\sinh(\para/2)}
  \\
  &\times
  \left(\vphantom{\left(\frac{\para}{4abN}\right)^{j}}
    \tau_{a,b}(\para/2)
    +
    \sum_{j=1}^{\infty}
    \frac{\tau_{a,b}^{(2j)}(\para/2)}{j!}
    \left(\frac{\para}{4abN}\right)^{j}
  \right.
  \\
  &\quad\quad
  \left.
    +
    \sum_{k}
    (-1)^{k+1}
    \frac{4\sin(k\pi/a)\sinh(k\pi/b)\sqrt{-\pi N}}{\sqrt{ab\para}}
    \exp
    \left(
      N\left(\frac{k^2\pi^2}{ab\para}+k\pi\sqrt{-1}-\frac{ab\para}{4}\right)
    \right)
  \right)
  \\
  =&
  \frac{e^{(ab-a/b-b/a)\para/(4N)}}{2\sinh(\para/2)}
  \left(\vphantom{\left(\frac{\para}{4abN}\right)^{j}}
    \tau_{a,b}(\para/2)
    +
    \sum_{k}
    (-1)^{k+1}A_k(\para;N)
    +
    \sum_{j=1}^{\infty}
    \frac{\tau_{a,b}^{(2j)}(\para/2)}{j!}
    \left(\frac{\para}{4abN}\right)^{j}
  \right).
\end{split}
\end{equation*}
\par
Now we consider the range of $k$.
\par
We observe that $C_{\para}$ crosses the imaginary axis at $\sqrt{-1}(\Im\para-\Re\para\tan\varphi)/2$, and when $\varphi$ increases from $(\arg\para)/2-\pi/4$ to $(\arg\para)/2+\pi/4$, the crossing point goes from $\sqrt{-1}\left(\Im\para-\Re\para\tan\bigl((\arg\para)/2-\pi/4\bigr)\right)$ to $\sqrt{-1}\left(\Im\para-\Re\para\tan\bigl((\arg\para)/2+\pi/4\bigr)\right)$ downwards (upwards, respectively) if $0\le\arg\para<\pi/2$ (if $\pi/2<\para\le\pi$, respectively).
Note that if $\pi/2<\para\le\pi$, $C_{\para}$ can be parallel to the imaginary axis but we avoid this.
Since
\begin{equation*}
\begin{split}
  \tan^2((\arg\para)/2\pm\pi/4)
  &=
  \frac{1-\cos(\arg\para\pm\pi/2)}{1+\cos(\arg\para\pm\pi/2)}
  \\
  &=
  \frac{1\pm\sin(\arg\para)}{1\mp\sin(\arg\para)}
  \\
  &=
  \frac{\bigl(1\pm\sin(\arg\para)\bigr)^2}{\cos^2(\arg\para)}
  \\
  &=
  \frac{(|\para|\pm\Im\para)^2}{(\Re\para)^2},
\end{split}
\end{equation*}
we have
\begin{equation*}
  \tan\bigl((\arg\para)/2\pm\pi/4\bigr)
  =
  \dfrac{\Im\para\pm|\para|}{\Re\para}.
\end{equation*}
and
\begin{equation*}
  \Im\para-\Re\para\tan\bigl((\arg\para)/2\pm\pi/4\bigr)
  =
  \mp|\para|.
\end{equation*}
So if $0\le\arg\para<\pi/2$, then the crossing point is between $-\sqrt{-1}|\para|$ and $\sqrt{-1}|\para|$, and  $k$ runs over integers that are not multiples of $a$ or $b$ with $1\le k\le M$ for any integer $M$ satisfying $0<M<ab|\para|/(2\pi)$.
If $\pi/2<\para\le\pi$, then the crossing point is above $\sqrt{-1}|\para|$ or below $-\sqrt{-1}|\para|$, and $k$ runs over all integers that are not multiples of $a$ or $b$ with $1\le k\le M'$ for any integer $M'$ with $M'>ab|\para|/(2\pi)$.
\par
So when $0\le\arg\para<\pi/2$, we have
\begin{multline}\label{eq:quadrant1}
  J_N\bigl(T(a,b);\exp(\para/N)\bigr)
  \\
  =
  \frac{e^{(ab-a/b-b/a)\para/(4N)}}{2\sinh(\para/2)}
  \left(
    A_0(\para)
    +
    \sum_{\stackrel{1\le k\le M,}{a\nmid k,b\nmid k}}
    (-1)^{k+1}A_k(\para;N)
    +
    \sum_{j=1}^{\infty}
    \frac{\tau_{a,b}^{(2j)}(\para/2)}{j!}
    \left(\frac{\para}{4abN}\right)^{j}
  \right)
\end{multline}
for any integer $M$ with $0<M<ab|\para|/(2\pi)$.
When $\pi/2<\arg\para\le\pi$, we have
\begin{multline}\label{eq:quadrant2}
  J_N\bigl(T(a,b);\exp(\para/N)\bigr)
  \\
  =
  \frac{e^{(ab-a/b-b/a)\para/(4N)}}{2\sinh(\para/2)}
  \left(
    A_0(\para)
    +
    \sum_{\stackrel{1\le k\le M',}{a\nmid k,b\nmid k}}
    (-1)^{k+1}A_k(\para;N)
    +
    \sum_{j=1}^{\infty}
    \frac{\tau_{a,b}^{(2j)}(\para/2)}{j!}
    \left(\frac{\para}{4abN}\right)^{j}
  \right).
\end{multline}
for any integer $M'$ with $M'>ab|\para|/(2\pi)$.
\par
Note that since the real part of $S_k(\para)/\para$ is
\begin{equation*}
  \left(\frac{k^2\pi^2}{ab|\para|^2}-\frac{ab}{4}\right)
  \Re\para,
\end{equation*}
the real part of the coefficient of $N$ in the exponent in $A_k(\para;N)$ is positive if and only if $\Re\para>0$ and $k>ab|\para|/(2\pi)$, or $\Re\para<0$ and $k<ab|\para|/(2\pi)$, negative if and only if $\Re\para>0$ and $k<ab|\para|/(2\pi)$, or $\Re\para<0$ and $k>ab|\para|/(2\pi)$, and zero if and only if $k=ab|\para|/(2\pi)$.
\par
Therefore in \eqref{eq:quadrant1} we can ignore all the $k$ since $A_k(\theta;N)$ decays exponentially, and in \eqref{eq:quadrant2} we can ignore $k$ with $k>ab|\para|/(2\pi)$.
Noting that if $a$ or $b$ divides $k$ then $A_k(\para;N)=0$, we finally have
\begin{multline*}
  J_N\bigl(T(a,b);\exp(\para/N)\bigr)
  \\
  \sim
  \frac{e^{(ab-a/b-b/a)\para/(4N)}}{2\sinh(\para/2)}
  \left(
    \tau_{a,b}(\para/2)
    +
    \sum_{j=1}^{\infty}
    \frac{\tau_{a,b}^{(2j)}(\para/2)}{j!}
    \left(\frac{\para}{4abN}\right)^{j}
  \right)
\end{multline*}
when $\Re{\para}>0$ and
\begin{multline*}
  J_N\bigl(T(a,b);\exp(\para/N)\bigr)
  \\
  \sim
  \frac{e^{(ab-a/b-b/a)\para/(4N)}}{2\sinh(\para/2)}
  \left(
    \tau_{a,b}(\para/2)
    +
    \sum_{k=1}^{\floor{ab|\para|/(2\pi)}}
    (-1)^{k+1}A_k(\para;N)
    +
    \sum_{j=1}^{\infty}
    \frac{\tau_{a,b}^{(2j)}(\para/2)}{j!}
    \left(\frac{\para}{4abN}\right)^{j}
  \right).
\end{multline*}
when $\Re{\para}<0$.
\end{proof}
\begin{rem}
When $\Re{\para}<0$ and $|\para|<2\pi/(ab)$ there is no $A_k(\para;N)$ term.
When $\Re{\para}<0$ and $|\para|=2\pi/(ab)$ the term $A_1(\para;N)$ oscillates.
\end{rem}
\begin{proof}[Proof of Theorem~\ref{thm:asymptotic} for purely imaginary $\para$.]
If $\para$ is purely imaginary, we have already shown the following formulas in \cite[Proposition~3.2]{Hikami/Murakami:COMCM2008}.
\par
If $\Re\para=0$ and $\para/2\not\in\mathcal{P}$, then we have
\begin{equation*}
\begin{split}
  &J_N\bigl(T(a,b);\exp(\para/N)\bigr)
  \\
  \sim&
  \frac{e^{(ab-a/b-b/a)\para/(4N)}}{2\sinh(\para/2)}
  \left(
    \tau_{a,b}(\para/2)
    +
    \sum_{k=1}^{\floor{ab|\para|/(2\pi)}}
    (-1)^{k+1}A_{k}(\para;N)
    +
    \sum_{j=1}^{\infty}
    \frac{\tau_{a,b}^{(2j)}(\para/2)}{j!}
    \left(\frac{\para}{4abN}\right)^{j}
  \right).
\end{split}
\end{equation*}
If $\para/2$ is in $\mathcal{P}$ but not an integral multiple of $2\pi\sqrt{-1}$, then we have
\begin{equation*}
\begin{split}
  &J_N\bigl(T(a,b);\exp(\para/N)\bigr)
  \\
  \sim&
  \frac{e^{(ab-a/b-b/a)\para/(4N)}}{2\sinh(\para/2)}
  \left(\vphantom{\left(\frac{\para}{4abN}\right)^{j}}
    \tau_{a,b}^{(0)}(\para/2)
    +
    \frac{1}{2}
    (-1)^{ab|\para|/(2\pi)}
    A_{ab|\para|/(2\pi)}(\para;N)
  \right.
  \\
  &\left.
    \quad+
    \sum_{k=1}^{ab|\para|/(2\pi)-1}
    (-1)^{k+1}A_{k}(\para;N)
    +
    \sum_{j=1}^{\infty}
    \frac{\tau_{a,b}^{(2j)}(\para/2)}{j!}
    \left(\frac{\para}{4abN}\right)^{j}
  \right),
\end{split}
\end{equation*}
where $\tau_{a,b}^{(0)}(\para/2)$ means the constant term of the Laurent expansion of $\tau_{a,b}(z)$ around $z=\para/2$.
This completes the proof.
\end{proof}
\section{A topological interpretation of the asymptotic behavior.}
\label{sec:topology}
In this section we study a topological interpretation of the term $A_k(\para;N)$ ($k\ge1$).
Given a positive integer $k$ that is not a multiple of $a$ nor $b$, we associate a pair of integers $(\alpha,\beta)$ as described in \S\ref{sec:character_variety}.
\subsection{A topological interpretation of $S_k(\para)$.}
Let $\rho_{\alpha,\beta}$ be an irreducible representation of $\pi_1\bigl(S^3\setminus{T(a,b)}\bigr)$ at $SL(2;\C)$ which is in the component of the character variety indexed by $(\alpha,\beta)$.
\par
The fundamental group of $S^3\setminus{T(a,b)}$ has a presentation $\pi_1\bigl(S^3\setminus{T(a,b)}\bigr)=\langle x,y\mid x^a=y^b\rangle$.
Then the longitude $\lambda$ can be expressed as $\lambda=x^{a}\mu^{-ab}$, where $\mu$ is the meridian.
Up to conjugation, we may assume that the images of $\mu$ and $\lambda$ are as follows.
\begin{align*}
  \rho_{\alpha,\beta}(\mu)
  &=
  \begin{pmatrix}
    \mathfrak{m} & \ast \\
    0            & \mathfrak{m}^{-1}
  \end{pmatrix},
  \\
  \rho_{\alpha,\beta}(\lambda)
  &=
  \begin{pmatrix}
    \mathfrak{l} & \ast \\
    0            & \mathfrak{l}^{-1}
  \end{pmatrix}.
\end{align*}
Since $\lambda\mu^{ab}=x^a$ and $x^a=\pm I$ (see for example \cite[Lemma~2.2]{Munoz:REVMC2009}), we have $\mathfrak{l}=\pm\mathfrak{m}^{-ab}$.
\par
If we put $\mathfrak{m}=\exp(u/2)$ and $\mathfrak{l}=-\exp(v/2)$, we have $\exp(v/2)=\pm\exp(-abu/2)$.
Therefore $v$ can be expressed (modulo $2\pi\sqrt{-1}$) in terms of $u$.
We choose $-ab(u+2\pi\sqrt{-1})+2(k-1)\pi\sqrt{-1}$ as such an expression and denote it by $v_k(u)$.
\par
Using the pair $(u,v_k(u))$ with $u:=\para-2\pi\sqrt{-1}$, we can prove that the function $\CS_{u,v_k(u)}([\rho_{\alpha,\beta}])$ defined in \S\ref{sec:Chern_Simons_invariant} can be expressed in terms of $S_k(\para)$.
\begin{thm}\label{thm:CS}
Let $\rho_{\alpha,\beta}$ be an irreducible representation such that $[\rho_{\alpha,\beta}]$ is in the component of $X\bigl(S^3\setminus{T(a,b)}\bigr)$ indexed by $(\alpha,\beta)$.
If we put $v_k(u):=-ab(u+2\pi\sqrt{-1})+2(k-1)\pi\sqrt{-1}$, then we have
\begin{equation*}
  \CS_{u,v_k(u)}\left([\rho_{\alpha,\beta}]\right)
  =
  S_k(\para)-\pi\sqrt{-1}u-\frac{uv_k(u)}{4}
\end{equation*}
with $u:=\para-2\pi\sqrt{-1}$, that is, the following equality holds.
\begin{equation*}
  \cs_{T(a,b)}([\rho_{\alpha,\beta}])
  =
  \left[
    \frac{u}{4\pi\sqrt{-1}},\frac{v_k(u)}{4\pi\sqrt{-1}};
    \exp
    \left(
      \frac{2}{\pi\sqrt{-1}}
      \left(S_k(\para)-\pi\sqrt{-1}u-\frac{uv_k(u)}{4}\right)
    \right)
  \right].
\end{equation*}
\end{thm}
\begin{proof}
From a formula by Dubois and Kashaev \cite[Proposition~4]{Dubois/Kashaev:MATHA2007} we have
\begin{multline}\label{eq:Dubois_Kashaev}
  \cs_{T(a,b)}([\rho_{\alpha,\beta}])
  \\
  =
  \left[
    \frac{u}{4\pi\sqrt{-1}},
    \frac{1}{2}-\frac{abu}{4\pi\sqrt{-1}};
    \exp
    \left(
      -8\pi\sqrt{-1}
      \left(
        \frac{(\beta ad+\varepsilon\alpha bc)^2}{4ab}-\frac{u}{8\pi\sqrt{-1}}
      \right)
    \right)
  \right],
\end{multline}
where integers $c$ and $d$ are chosen so that $ad-bc=1$, and $\varepsilon=\pm1$.
Note that we are using the $PSL(2;\C)$ normalization and so we need to multiply the exponent in the third entry by $-4$.
\par
Changing the coordinate by using \eqref{eq:coordinate_change}, we have
\begin{equation*}
\begin{split}
  &\cs_{T(a,b)}([\rho_{\alpha,\beta}])
  \\
  =&
  \left[
    \frac{u}{4\pi\sqrt{-1}},
    \frac{1}{2}-\frac{abu}{4\pi\sqrt{-1}}
    +\frac{k-ab-2}{2};
    \exp
    \left(
      \frac{-2k^2\pi\sqrt{-1}}{ab}+u+(k-ab-2)u
    \right)
  \right]
  \\
  =&
  \left[
    \frac{u}{4\pi\sqrt{-1}},
    \frac{v_k(u)}{4\pi\sqrt{-1}};
    \exp
    \left(\frac{-2k^2\pi\sqrt{-1}}{ab}-abu+(k-1)u\right)
  \right]
  \\
  =&
  \left[
    \frac{u}{4\pi\sqrt{-1}},
    \frac{v_k(u)}{4\pi\sqrt{-1}};
    \exp
  \left(
    \frac{2}{\pi\sqrt{-1}}
    \left(
      S_k(\para)-\pi\sqrt{-1}u-\frac{uv_k(u)}{4}
    \right)
  \right)
  \right],
\end{split}
\end{equation*}
where the first equality follows since $k^2=(\beta ad+\varepsilon\alpha bc)^2\pmod{ab}$ and the last equality follows since $u=\para-2\pi\sqrt{-1}$.
Note that the choice of $\varepsilon$ does not matter here.
Note also that even if we change the definition of $(\alpha,\beta)$, the equality still holds (Remark~\ref{rem:def_alpha_beta}).
\end{proof}
\par
Since $v_k(u)=2\frac{d\,S_k(\para)}{d\,\para}\Bigr|_{\para:=u+2\pi\sqrt{-1}}-2\pi\sqrt{-1}$, we note that $\CS_{u,v_k(u)}([\rho_{\alpha,\beta}])$ can be determined by $S_k(\para)$.
\subsection{A topological interpretation of $T_k$.}
We can show that $T_k$ is the twisted Reidemeister torsion associated with the meridian $\mu$.
\begin{lem}
Let $\rho_{\alpha,\beta}$ be an irreducible representation $\pi_1\bigl(S^3\setminus{T(a,b)}\bigr)\to SL(2;\C)$ whose character belongs to the component indexed by $(\alpha,\beta)$ that is determined by $k$ as described in \S~\ref{sec:character_variety}.
Then (up to a sign) the Reidemeister torsion $\mathbb{T}_{\mu}^{T(a,b)}(\rho_{\alpha,\beta})$ associated with the meridian $\mu$ is given by
\begin{equation*}
  \mathbb{T}_{\mu}^{T(a,b)}(\rho_{\alpha,\beta})
  =
  \pm
  \frac{16}{ab}
  \sin^2\left(\frac{\pi\alpha}{a}\right)\sin^2\left(\frac{\pi\beta}{b}\right).
\end{equation*}
\end{lem}
\begin{proof}
If an irreducible representation $\rho_{\alpha,\beta}$ is in the component indexed by $(\alpha,\beta)$, Dubois \cite[6.2]{Dubois:CANMB2006} proved that the twisted Reidemeister torsion $\mathbb{T}_{\lambda}^{T(a,b)}(\rho_{\alpha,\beta})$ associated with the {\em longitude} $\lambda$ is given by
\begin{equation}\label{eq:Dubois}
  \mathbb{T}_{\lambda}^{T(a,b)}(\rho_{\alpha,\beta})
  =
  \frac{16}{a^2b^2}
  \sin^2\left(\frac{\alpha\pi}{a}\right)
  \sin^2\left(\frac{\beta\pi}{b}\right).
\end{equation}
\par
From Remark (ii) to \cite[Th{\'e}or{\`e}me~4.1]{Porti:MAMCAU1997}, we have
\begin{equation}\label{eq:torsion_base_change}
  \mathbb{T}_{\mu}(\rho_{\alpha,\beta})
  =
  \pm
  \frac{\partial\,v}{\partial\,u}
  \mathbb{T}_{\lambda}(\rho_{\alpha,\beta})
\end{equation}
for an irreducible representation $\rho$.
Here $u$ and $v$ are parameters as described in the previous subsection.
Note that we are using cohomological Reidemeister torsion and Porti uses homological one.
So our torsion is the inverse of the torsion used in \cite{Porti:MAMCAU1997}.
\par
As in the previous subsection $v=-abu+2n\pi\sqrt{-1}$ for a constant $n\in\Z$.
So we have $\partial\,v/\partial\,u=-ab$ and the lemma follows.
\end{proof}
\begin{rem}
In \cite{Porti:MAMCAU1997} Porti uses the twisted homology instead of the twisted cohomology.
So the Reidemeister torsion is the inverse of ours.
The authors thank J.~Dubois for pointing out this.
\end{rem}
Since $\sin^2(k\pi/a)\sin^2(k\pi/b)=\sin^2(\alpha\pi/a)\sin^2(\beta\pi/b)$ (Remark~\ref{rem:def_alpha_beta}), we have
\begin{equation*}
  \mathbb{T}_{\mu}^{T(a,b)}(\rho_{\alpha,\beta})
  =
  \pm
  \frac{16}{ab}
  \sin^2\left(\frac{k\pi}{a}\right)
  \sin^2\left(\frac{k\pi}{b}\right).
\end{equation*}
Since $T_{k}$ is always positive, we have the following theorem.
\begin{thm}
Let $\rho_{\alpha,\beta}$ be an irreducible representation of $\pi_1\bigl(S^3\setminus{T(a,b)}\bigr)$ at $SL(2;\C)$, which is in the component indexed by $(\alpha,\beta)$ that is associated with an integer $k$ as in \S{\rm\ref{sec:character_variety}}.
\par
Then $T_{k}$ equals the absolute value of the twisted Reidemeister torsion of $\rho_{\alpha,\beta}$ associated with the meridian, that is, we have
\begin{equation*}
  T_{k}
  =
  \left|\mathbb{T}_{\mu}^{T(a,b)}(\rho_{\alpha,\beta})\right|.
\end{equation*}
\end{thm}
\subsection{Remaining factor}
There remains a strange factor $2\sinh(\para/2)$ in the asymptotic expression.
Recall that we normalize the colored Jones polynomial so that its value for the unknot is one.
Another (more natural in physics) normalization is to put the value for the empty link to be one.
In this normalization the colored Jones polynomial of the unknot is $[N]=\bigl(q^{N/2}-q^{-N/2}\bigr)/\bigl(q^{1/2}-q^{-1/2}\bigr)$.
\par
Then the factor $2\sinh(\para/2)$ comes from the following asymptotic expansion at $q=\exp(\para/N)$ of $\log[N]$.
\begin{equation*}
  \frac{\sinh(\para/2)}{\sinh(\para/(2N))}
  \sim
  2\sinh(\para/2)\frac{N}{\para}
  -
  \frac{\sinh(\para/2)}{12}\left(\frac{N}{\para}\right)^{-1}
  +\cdots.
\end{equation*}
\section{Speculation}
\label{sec:speculation}
Combining the results of Section~\ref{sec:topology}, for a torus knot $K$ and an appropriately chosen parameter $\para$ we have
\begin{multline*}
  \lim_{N\to\infty}
  \left\{\vphantom{\left(\frac{N}{\para}\right)^{1/2}}
    J_N(K;\exp(\para/N))
    \frac{2\sinh(\para/2)}{\nu(\para/N)}
  \right.
  \\
  \left.
    -
    \sqrt{-\pi}
    \sum_{k}
    (-1)^{k+1}
    \left(\frac{N}{\para}\right)^{1/2}
    \exp\left(S_{k}(\para)\frac{N}{\para}\right)
    \left(\mathbb{T}_{\lambda}^{K}(\rho_k)\right)^{1/2}
  \right\}
  =
  \frac{2\sinh(\para/2)}{\Delta(K;\exp{\para})}
\end{multline*}
where $\nu(x)$ is a function that converges to $1$ when $x\to0$, $k$ runs over some irreducible components of the character variety $X(S^3\setminus{K})$, $\rho_k$ is an irreducible representation in the component indexed by $k$, and $S_{k}(\para)$ determines the $SL(2;\C)$ Chern--Simons invariant $\CS_{u,v_k(u)}([\rho_k])$ as in Theorem~\ref{thm:CS}.
We expect a similar formula for a general knot.
\par
Here we just give an observation about the figure-eight knot.
\par
In \cite{Murakami/Yokota:JREIA2007} Yokota and the second author proved that for the figure-eight knot $E$, the following holds.
\begin{thm}[\cite{Murakami/Yokota:JREIA2007}]
There exists a neighborhood $U$ of $0$ in $\C$ such that for any $u\in(U\setminus\pi\sqrt{-1}\Q)\cup\{0\}$, the following limit exists
\begin{equation*}
  (u+2\pi\sqrt{-1})
  \lim_{N\to\infty}
  \frac{\log{J_N\bigl(E;\exp((u+2\pi\sqrt{-1})/N)\bigr)}}{N}.
\end{equation*}
Moreover if we denote the limit by $H(u)$ and put $v(u):=2\frac{d\,H(u)}{d\,u}-2\pi\sqrt{-1}$, then $H(u)-\pi\sqrt{-1}u-uv(u)/4$ coincides with $\CS_{u,v(u)}([\rho])$, where $\rho$ is the representation of $\pi_1(S^3\setminus{E})$ at $SL(2;\C)$ sending the meridian to $\begin{pmatrix}\exp(u/2)&\ast\\0&\exp(-u/2)\end{pmatrix}$ and the longitude to $\begin{pmatrix}-\exp(v(u)/2)&\ast\\0&-\exp(-v(u)/2)\end{pmatrix}$ up to conjugate.
\end{thm}
The $SL(2;\C)$ character variety of $S^3\setminus{E}$ has two connected components, the abelian one and the non-abelian one.
\par
Non-abelian representations can be calculated explicitly by using the technique described in \cite{Riley:QUAJM31984} (see also \cite[\S3.1]{Murakami:ACTMV2008}).
Let $\rho_{m\pm}$ be the non-abelian representation of $\pi_1(S^3\setminus{E})$ at $SL(2;\C)$ sending the meridian to
\begin{equation*}
  \begin{pmatrix}
    m^{1/2}&1 \\
    0      &m^{-1/2}
  \end{pmatrix}
\end{equation*}
and the longitude to
\begin{equation*}
  \begin{pmatrix}
    \ell(m)^{\pm1} &
    \left(m^{1/2}+m^{-1/2}\right)
    \sqrt{(m+m^{-1}+1)(m+m^{-1}-3)} \\
    0 &
    \ell(m)^{\mp1}
  \end{pmatrix},
\end{equation*}
where
\begin{equation}\label{eq:ell}
  \ell(m)
  :=
  \frac{m^2-m-2-m^{-1}+m^{-2}}{2}
  +
  \frac{m-m^{-1}}{2}
  \sqrt{(m+m^{-1}+1)(m+m^{-1}-3)}.
\end{equation}
See \cite[\S3.1]{Murakami:ACTMV2008} for details.
Note that the pair $(m,\ell(m))$ is a zero of the $A$-polynomial
\begin{equation}\label{eq:A_fig8}
  \ell-\left(m^2-m-2-m^{-1}+m^{-2}\right)+\ell^{-1}.
\end{equation}
\begin{rem}
Equation~(3.8) in \cite{Murakami:ACTMV2008} is mistyped.
It should be read as
\begin{equation*}
  \ell-\left(m^2-m-2-m^{-1}+m^{-2}\right)+\ell^{-1}=0.
\end{equation*}
The authors thank E.~Witten, who pointed out this.
\end{rem}
In \cite[\S6.3]{Dubois:CANMB2006} Dubois proves that the twisted Reidemeister torsion $\mathbb{T}_{\lambda}^{E}(\rho_{m\pm})$ associated with the longitude $\lambda$ is given by
\begin{equation*}
  \mathbb{T}_{\lambda}^{E}(\rho_{m\pm})
  =
  \frac{1}{\sqrt{17+4\Tr(\rho_{m\pm}(\lambda))}}
  =
  \frac{1}{2m+2m^{-1}-1}.
\end{equation*}
See also \cite[\S4.5]{Porti:MAMCAU1997} and \cite{Dubois/Huynh/Yamaguchi:JKNOT2009}.
\par
Therefore from \eqref{eq:torsion_base_change} the twisted Reidemeister torsion associated with the meridian $\mu$ is
\begin{equation*}
  \mathbb{T}_{\mu}^{E}(\rho_{m\pm})
  =
  \pm
  \frac{\partial\,v}{\partial\,u}
  \times
  \frac{1}{2m+2m^{-1}-1}.
\end{equation*}
Since in this case $e^{u/2}=m^{1/2}$ and $e^{v/2}=-\ell(m)^{\pm1}$, we have
\begin{equation*}
\begin{split}
  \frac{\partial\,v}{\partial\,u}
  &=
  \frac{\pm\partial\bigl(2\log{\ell(m)\bigr)/\partial\,m}}
       {\partial(\log{m})/\partial\,m}
  \\
  &=
  \frac{\pm2m\,d\,\ell(m)/d\,m}{\ell(m)}.
\end{split}
\end{equation*}
Since the pair $(m,\ell(m))$ is a zero of the A-polynomial, differentiating \eqref{eq:A_fig8} by $m$ we have
\begin{equation*}
  \frac{d\,\ell(m)}{d\,m}
  =
  \frac{2m-1+m^{-2}-2m^{-3}}{1-\ell(m)^{-2}}.
\end{equation*}
Therefore we finally have
\begin{equation*}
\begin{split}
  \mathbb{T}_{\mu}^{E}(\rho_{m\pm})
  &=
  \frac{\pm2}{\sqrt{(m+m^{-1}+1)(m+m^{-1}-3)}}.
\end{split}
\end{equation*}
\par
By some computer calculations the following formula seems to hold.
\begin{multline*}
  \lim_{N\to\infty}
  \left\{
    J_{N}\bigl(E;\exp(\para/N)\bigr)
    \frac{2\sinh(\para/2)}{\nu(\para/N)}
    -
    \sqrt{-\pi}
    \exp\left(H(u)\frac{N}{\para}\right)
    \left(\frac{N}{\para}\right)^{1/2}
    \sqrt{\mathbb{T}_{\mu}^{E}(\rho_{m\pm})}
  \right\}
  \\
  =
  \frac{2\sinh(\para/2)}{\Delta(E;\exp{\para})}
\end{multline*}
where $\nu(x)$ is a function with $\lim_{x\to0}\nu(x)=1$.
\par
For a hyperbolic knot $K$, we expect a similar formula.
\begin{multline*}
  \lim_{N\to\infty}
  \left\{
    J_{N}\bigl(K;\exp(\para/N)\bigr)
    \frac{2\sinh(\para/2)}{\nu(\para/N)}
    -
    \sqrt{-\pi}
    \exp\left(H(u)\frac{N}{\para}\right)
    \left(\frac{N}{\para}\right)^{1/2}
    \sqrt{\mathbb{T}_{\mu}^{K}(\rho)}
  \right\}
  \\
  =
  \frac{2\sinh(\para/2)}{\Delta(E;\exp{\para})},
\end{multline*}
where $\nu(x)$ is a function with $\lim_{x\to0}\nu(x)=1$ and we put $u:=\para-2\pi\sqrt{-1}$.
Moreover  $\rho$, $H(u)$ and $\mathbb{T}_{\mu}^{K}(\rho)$ satisfy the following properties.
Put $v(u):=2\frac{d\,H(u)}{d\,u}-2\pi\sqrt{-1}$.
\begin{itemize}
\item
$\rho\colon\pi_1(S^3\setminus{K})\to SL(2;\C)$ sends the meridian to $\begin{pmatrix}\exp(u/2)&\ast\\0&\exp(-u/2)\end{pmatrix}$ and the longitude to $\begin{pmatrix}-\exp(v(u)/2)&\ast\\0&-\exp(-v(u)/2)\end{pmatrix}$ up to conjugate.
\item
$H(u)-\pi\sqrt{-1}u-uv(u)/4$ coincides with $\CS_{u,v(u)}([\rho])$.
\item
$\mathbb{T}_{\mu}^{K}(\rho)$ is the twisted Reidemeister torsion of $\rho$ associated with the meridian.
\end{itemize}
\bibliography{mrabbrev,hitoshi}
\bibliographystyle{amsplain}
\end{document}